\documentclass[final,12pt,times]{elsarticle}
\biboptions{sort&compress}

\usepackage{amssymb}
\usepackage{amsthm}
\usepackage[intlimits]{amsmath}
\usepackage{upref}

\numberwithin{equation}{section}

\newtheorem{thm}{Theorem}[section]
\newtheorem{lem}[thm]{Lemma}
\newtheorem{cor}[thm]{Corollary}
\newdefinition{rmk}{Remark}
\newproof{pf}{Proof}
\newproof{pot}{Proof of Theorem \ref{thm2}}

\newcommand{\diag}{\operatorname{diag}}
\newcommand{\mb}{\mathbf}
\newcommand{\mbb}{\mathbb}
\newcommand{\re}{\operatorname{Re}}
\newcommand{\tr}{\operatorname{tr}}

\journal{Advances in Applied Mathematics}

\begin{document}

\begin{frontmatter}



\title{Kibble--Slepian Formula and Generating Functions for $2$D Polynomials}


\author[ksu,ucf]{Mourad E. H. Ismail\fnref{fn1}}
\ead{mourad.eh.ismail@gmail.com}

\author[nafu]{Ruiming Zhang\corref{cor1}\fnref{fn2}}
\ead{ruimingzhang@yahoo.com}

\cortext[cor1]{Corresponding author}
\fntext[fn1]{Research supported by a grant from the DSFP program at King Saud University and by the National Plan for Science, Technology and Innovation (MAARIFAH), King Abdelaziz City for Science and Technology, Kingdom of Saudi Arabia, Award Number 14-MAT623-02.}
\fntext[fn2]{Research partially supported by National Science Foundation of China, Grant No. 11371294.}

\address[ksu]{Department of Mathematics, King Saud University, Riyadh, Saudi Arabia}
\address[ucf]{Department of Mathematics, University of Central Florida, Orlando, Florida 32816, USA}
\address[nafu]{College of Science, Northwest A\&F University, Yangling, Shaanxi 712100, P. R. China}

\begin{abstract}
We prove a generalization of the Kibble--Slepian formula (for Hermite polynomials) and its unitary analogue involving the $2$D Hermite polynomials recently proved in \cite{Ism4}. We derive integral representations for the  $2$D Hermite polynomials which are of independent interest. Several new generating functions for $2$D $q$-Hermite polynomials will also be given.
\end{abstract}


\end{frontmatter}


\section{Introduction}\label{sec1}

The complex Hermite polynomials $\left\{H_{m,n}\left(z_1,z_2\right)\right\}_{m,n=0}^{\infty}$ may be defined by
\begin{equation}
H_{m,n}\left(z_1,z_2\right)=\sum_{k=0}^{m\wedge n}(-1)^{k}k!\,\binom{m}{k}\binom{n}{k}\,z_1^{m-k}z_2^{n-k}.
\label{eq:1.1}
\end{equation}
The polynomials $\left\{H_{m,n}\left(z,\overline{z}\right)\right\}_{m,n=0}^{\infty}$ are orthogonal on $\mbb{R}^2$ with respect to $e^{-x^2-y^2}$ and have the generating function
\begin{equation}
\sum_{m,\,n=0}^{\infty}\,H_{m,n}\left(z_1,z_2\right)\frac{u^m\,v^n}{m!\,n!}=e^{uz_1+vz_2-uv}.
\label{eqGFHmn}
\end{equation}
They seem to have been considered first by Ito \cite{Ito} in his study of complex multiple Wiener integrals. Recently they were used in \cite{Ali:Bag:Hon} to study Landau levels and were applied in \cite{Thi:Hon:Krz} to
coherent states, and in \cite{Wun,Wun2} to quantum optics and quasiprobabilities, respectively. See also \cite{Gha,Gha2,Cot:Gaz:Gor}. The reference \cite{Int:Int} deals with the spectral properties of the Cauchy transform and the polynomials $\left\{H_{m,n}(z,\overline{z})\right\}$ also appear in this context. The polynomials $\left\{H_{m,n}\left(z,\overline{z}\right)\right\}_{m,n=0}^{\infty}$ are essentially the same polynomials as in the monograph \cite[(2.6.6)]{Dun:Xu} by Dunkl and Xu.

The Kibble--Slepian formula is Equation \eqref{eq:1.3} below. It was first proved by Kibble in 1945 \cite{Kib} and later by Slepian \cite{Sle}. Louck \cite{Lou} gave a proof using Boson operators while Foata \cite{Foa}
gave a purely combinatorial proof. Each proof brings in a new point of view.

\begin{thm}\label{thm:thm1}
Let $S=\left(s_{j,k}\right)_{j,k=1}^N$ be an $N\times N$ real symmetric matrix with the Frobenius norm
\begin{equation}
\|S\|^2=\sum_{j,k=1}^N\left|s_{j,k}\right|^2.
\label{eq:1.2}
\end{equation}
Assume that $\|S\|<1$, $I_N$ being an identity matrix of size $N$, and $X$ being an $N\times1$ matrix. Then
\begin{equation}
\label{eq:1.3}
\begin{gathered}
\det\left(I_N+S\right)^{-\frac{1}{2}}\exp\left(X^TS\left(I_N+S\right)^{-1}X\right)\\
=\sum_{K}\left(\prod_{1\le m\le n\le N}\frac{\left(s_{m,n}\right)^{k_{m,n}}}
{2^{k_{m,n}}k_{m,n}!}\right)2^{-\tr\left(K\right)}H_{k_1}\left(x_1\right)\dotsb H_{k_N}\left(x_N\right),
\end{gathered}
\end{equation}
where $X=\left(x_1,x_2,\dotsc,x_N\right)^T$, $K=\left(k_{m,n}\right)_{m,n=1}^{N}$, $k_{m,n}=k_{n,m}$, $1\le m,n\le N$ and
\begin{equation}
\tr\left(K\right)=\sum_{j=1}^{N}k_{j,j},\quad k_{\ell}=k_{\ell,\ell}+\sum_{j=1}^{N}k_{\ell,j},\qquad\ell=1,\dots,N.
\label{eq:1.4}
\end{equation}
In \eqref{eq:1.3}, $\sum_{K}$ denotes the $\frac{n\left(n+1\right)}{2}$ fold sum over all nonnegative integers $k_{m,n}=0,1,\dotsc$ for all positive integers $m,n$ such that $1\le m\le n\le N$.
\end{thm}

It must be noted that the proofs by Louck \cite{Lou} and Slepian \cite{Sle} assume $S$ is symmetric and conclude that the expansion in \eqref{eq:1.3} holds for $\|S\|<1$. On the other hand the combinatorial version by Foata \cite{Foa} makes no assumptions on $S$ but assumes the diagonal elements $h_{j,j}$ vanish, and concludes that the expansion \eqref{eq:1.3} holds as a formal power series.

In \cite{Ism4} Ismail proved a similar theorem for the complex Hermite polynomials. His result is essentially the following theorem.

\begin{thm}\label{thm:thm2}
Let $W=\left(w_1,w_2,\cdots,w_{N}\right)^T$, and $H=\left(h_{m,n}\right)_{m,n=1}^N$ be an $N\times N$ Hermitian matrix with $\|H\|<1$ in Frobenius norm, and $I_N$ is an $N\times N$ identity matrix. Then
\begin{equation}
\label{eq:1.5}
\begin{gathered}
\det\left(I_N+H\right)^{-1}\exp\left(W^*H\left(I_N+H\right)^{-1}W\right) \\
=\sum_{K}\prod_{1\le m,n\le N}\frac{\left(h_{m,n}\right)^{k_{m,n}}}{k_{m,n}!}H_{r_1,c_1}
\left(\overline{w_1},w_1\right)\dotsm H_{r_N,c_N}\left(\overline{w_N},w_N\right),
\end{gathered}
\end{equation}
where $K=\left(k_{m,n}\right)_{m,n=1}^N$ is a general matrix with nonnegative integer entries, $c_n$ is the sum of the elements of $K$ in column $n$ and $r_m$ is the sum of the elements of $K$ in row $m$, that is
\begin{equation}
c_{n}=\sum_{j=1}^Nk_{j,n},\qquad r_{m}=\sum_{\ell=1}^Nk_{m,\ell}.
\label{eq:1.6}
\end{equation}
\end{thm}

In this paper we prove the following stronger result without the assumption that $H\in \mathbb{C}^{N\times N}$ is Hermitian.

\begin{thm}\label{thm2b}
Following the notations in Theorem \ref{thm:thm2}, we let $W=\left(w_1,\dotsc,w_N\right)^T\in\mbb{C}^N$, $H\in\mbb{C}^{N\times N}$, $\|H\|_{\infty}=\max\limits_{1\le j,\ell\le N}\left|h_{j,\ell}\right|$ and $B=\left\{H:\|H\|_{\infty}<\frac1N\right\}$. Then the series
$$
\sum_{K}\prod_{1\le m,n\le N}\frac{\left(h_{m,n}\right)^{k_{m,n}}}{k_{m,n}!}
H_{r_{1},c_{1}}\left(\overline{w_{1}},w_{1}\right)\dotsm H_{r_{N},c_{N}}\left(\overline{w_{N}},w_{N}\right)
$$
converges absolutely and uniformly for $W$ in any compact subset of $\mbb{C}^N$ and $H$ in any compact subset of $B$.
 	
Given $\delta_{j,k}>0$, $j,k=1,\dotsc,N$ and a Hermitian matrix $H_0=\left(h_{j,k}^{(0)}\right)_{j,k=1}^N\in B$, let
\begin{equation}
D\left(H_0,\delta\right)=\left\{H:\left|h_{j,j}-h_{j,j}^{(0)}\right|<\delta_{j,j},\left|u_{\ell,k}-u_{\ell,k}^{(0)}\right|
<\delta_{\ell,k},\left|v_{\ell,k}-v_{\ell,k}^{(0)}\right|<\delta_{\ell,k}\right\},
\label{eq:a-3}
\end{equation}
where $1\le j,k,\ell\le N$, $\ell<k$ and
\begin{equation}
u_{\ell,k}=\frac{h_{\ell,k}+h_{k,\ell}}{2},\,
v_{\ell,k}=\frac{h_{\ell,k}-h_{k,\ell}}{2i},\,
u_{\ell,k}^{(0)}=\frac{h_{\ell,k}^{(0)}+h_{k,\ell}^{(0)}}{2},\,
v_{\ell,k}^{(0)}=\frac{h_{\ell,k}^{(0)}-h_{k,\ell}^{(0)}}{2i}.
\label{eq:a-4}
\end{equation}
If $D\left(H_0,\delta\right)\subset B$, then
\begin{equation}
\begin{gathered}
\exp\left(W^*H\left(I_N+H\right)^{-1}W\right)=\det\left(I_N+H\right) \\
\times\sum_{K}\prod_{1\le m,n\le N}\frac{\left(h_{m,n}\right)^{k_{m,n}}}{k_{m,n}!}
H_{r_1,c_1}\left(\overline{w_1},w_1\right)\cdots H_{r_N,c_N}\left(\overline{w_N},w_N\right)
\end{gathered}
\label{eq:a-2}
\end{equation}
holds for all $W\in\mbb{C}^N$ and $H\in D\left(H_0,\delta\right)$. In particular, it is not hard to see that $D\left(\left(0\right)_{j,k=1}^{N},\left(\frac{1}{2N}\right)_{j,k=1}^{N}+\frac{I_{N}}{2N}\right)\subset B$.
\end{thm}

\begin{cor}\label{cor:cor1}
For $N\in\mbb{N}$, let $W=\left(\rho_1e^{i\theta_1},\dotsc,\rho_Ne^{i\theta_N}\right)^T$ that $\rho_m>0$, $\theta_m\in\mbb{R}$ for $m=1,\dotsc,N$ in \eqref{eq:1.5}, $H$, $I_N$, $K$, $c_m$ and $r_m$ are the same as in Theorem \ref{thm:thm2}, then
\begin{equation}
\label{eq:1.7}
\begin{gathered}
  \det\left(I_{N}+H\right)^{-1}
  \exp\left(W^{*}H\left(I_{N}+H\right)^{-1}W\right) \\
  =\sum_{K}\prod_{m=1}^{N}\prod_{n=1}^{N}
  \left(-h_{m,n}\right)^{k_{m,n}}\binom{c_{m}}{k_{1,m},\dotsc,k_{N,m}}
  \left(\rho_{m}e^{i\theta_{m}}\right)^{r_{m}-c_{m}}
  L_{c_{m}}^{\left(r_{m}-c_{m}\right)}\left(\rho_{m}^2\right),
\end{gathered}
\end{equation}
where $\left\{L_n^{(\alpha)}\left(x\right)\right\}$ are Laguerre polynomials. In particular, for $x,y>0$ and $\,|u|,|v|<\frac{|xy|}{4}$, we have
\begin{equation}
\begin{aligned} 
& \sum_{0\le j<k<\infty}\frac{\left(u^{j}v^{k}+u^{k}v^{j}\right)}{j!k!}C_{j}\left(k;x\right)C_{j}\left(k;y\right)\\
& =\frac{xy}{xy-uv}\exp\left(-\frac{xy\left(xuv-xy(u+v)+yuv\right)}{xy-uv}\right)\\
& -\frac{xy}{xy-uv}\exp\left(-\frac{uv(x^{2}+y^{2})}{xy-uv}\right)I_{0}\left(2\frac{\sqrt{uv(xy)^{3/2}}}{xy-uv}\right),
\label{eq:1.7a}
\end{aligned}
\end{equation}
where $C_{n}(x;a)$ is the $n$-th Charlier polynomial, $I_{\alpha}(z)$
is the Bessel function of first kind. 

\end{cor}

Ismail's proof in \cite{Ism4} assumes that $H$ is Hermitian and $\|H\|<1$ and proves that \eqref{eq:1.5} holds as a convergent power series in the variables $h_{j,k}$, $1\le j\le k\le N$. Later Ismail and Zeng \cite{Ism:Zen} found a combinatorial proof of Theorem \ref{thm:thm2} where $H$ is not necessary symmetric, but the power series in \eqref{eq:1.5} is a formal power series.

The purpose of this paper is to first prove Theorems \ref{thm:thm1} and \ref{thm:thm2} and Corollary \ref{cor:cor1} by using the integral representations
\begin{equation}
H_{n}\left(x\right)e^{-x^{2}}=\frac{\left(-2i\right)^{n}}{\sqrt{\pi}}
\int_{-\infty}^{\infty}t^{n}e^{-t^{2}+2ixt}dt,
\label{eq:1.8}
\end{equation}
and
\begin{equation}
e^{-z\overline{z}}H_{m,n}\left(z,\bar z\right)
=\frac{ i^{m+n}}{\pi}\int_{\mbb{R}^{2}}
w^{m}{\bar w}^{n}\exp\left\{-\left(r^{2}+s^{2}\right)
-2i\re\left(w \bar z\right)\right\} drds,
\label{eq:1.9}
\end{equation}
where $z=x+iy$ and $w=r+is$ such that $r,s,x,y\in\mbb{R}$. The representation \eqref{eq:1.8} is well-known, see for example, formula (4.6.41) in \cite{Ism}, while \eqref{eq:1.9} will be proved in \S\ref{sec3}. Our proof actually proves a stronger version of Theorems \ref{thm:thm1}--\ref{thm:thm2}, where $S$ and $H$ are not assumed to be symmetric and Hermitian, respectively.

It is important to note that the left-hand sides of the multilinear generating functions in Theorems \ref{thm:thm1}--\ref{thm:thm2} are positive, when $S$ and $H$ are real symmetric and Hermitian, respectively. They contain the Poisson kernels as the special cases when $N=2$ and the diagonal elements of the matrices involved are zero, \cite{Ism4}, \cite[\S 4.7]{Ism}. Carlitz \cite{Car} actually found the Poisson kernel for the $2D$ Hermite polynomials in 1978, $20$ years before \cite{Wun,Wun2}. He identified the $2D$ Hermite polynomials as special cases of a $3D$ system which he studied in detail but did not derive its orthogonality. Carlitz was not aware that his polynomials are the same as Ito's $2D$-Hermite polynomials. Carlitz did not elaborate on the orthogonality of his bivariate or trivariate polynomials.

Section \ref{sec2} contains the proofs of Theorems \ref{thm:thm1}--\ref{thm:thm2}. Section \ref{sec3} contains some new formal properties of the $2D$ Hermite polynomials. In our approach we treat $H_{m,n}\left(z_1,z_2\right)$ as a function of two independent complex variables and view the case $z_2=\overline{z_1}$ as a domain of orthogonality in $\mbb{C}^2$. In Section \ref{sec4} we derive several multiliear generating functions for the $2D$ $q$-Hermite polynomials we introduced in \cite{Ism:Zha}. We do not have a $q$-analogue of the Kibble--Slepian formula of Theorem \ref{thm:thm1} but the results in \S\ref{sec4} would be special cases of such formula. There
is no Kibble-Slepian formula known for the one variable $q$-Hermite polynomials either but their Poisson kernel is known.

\section{Proofs}\label{sec2}

We shall use the multidimensional Taylor series for functions mapping $\mbb{R}^N$ into $\mbb{R}$. For $\alpha=(\alpha_{1},\alpha_{2},\dots)$ such that $\alpha_{1},\,\alpha_{2},\dots$
are nonnegative integers, let $|\alpha|=\alpha_{1}+\alpha_{2}+\cdots$
, $\alpha!=\alpha_{1}!\cdot\alpha_{2}!\cdots$, and $\mb{x}^{\alpha}=x_{1}^{\alpha_{1}}\cdot x_{2}^{\alpha_{2}}\cdots$
. Additionally, for $n\in\mathbb{N}_{0}$ and $|\alpha|=n$ we let
$\binom{n}{\alpha}=\frac{n!}{\alpha!}$ and $D^{\alpha}f(\mb{x})=\frac{\partial^{|\alpha|}f(\mb{x})}{\partial x_{1}^{\alpha_{1}}\cdot\partial x_{2}^{\alpha_{2}}\cdots}$.
\begin{thm}\label{thmnDTaylor}
Assume that $f$ and all its partial derivatives of order $<m$ are differentiable at each point of an open set $S\subset\mbb{R}^n$. If $\mb{a}$ and $\mb{b}$ are two points of $S$ such that the line joining $\mb{a}$ and $\mb{b}$ is contained in $S$. We further let
\begin{equation}
f^{(k)}(\mb{x};\mb{t})=\sum_{|\alpha|=k}\binom{k}{\alpha}D^{\alpha}f(\mb{x})\mb{t}^{\alpha},
\end{equation}
then
\begin{equation}
f(\mb{b})=f(\mb{a})+\sum_{k=1}^{m-1}\frac{1}{k!}\,f^{(k)}(\mb{a};\mb{b}-\mb{a})
+\frac{1}{m!}\,f^{(m)}(\mb{z};\mb{b}-\mb{a}),
\end{equation}
for some $\mb{z}$ on the line segment joining $\mb{b}$ and $\mb{a}$.
\end{thm}

This is essentially Theorem 12.14 in \cite{Apo}.

\begin{lem}\label{lem:1}
Let $S=\left(s_{j,k}\right)_{j,k=1}^{N}$ be an $N\times N$ real symmetric matrix and $Y$ an $N\times1$ complex matrix, then,
\begin{equation}
\exp\left(-Y^{T}SY\right)=\sum_{K}\left(\prod_{1\le m\le n\le N}\frac{\left(-2s_{m,n}\right)^{k_{m,n}}}{k_{m,n}!}\right)2^{-\tr\left(K\right)}y_1^{k_1}y_2^{k_2}\cdots y_n^{k_n},
\label{eq:2.1}
\end{equation}
where $K,k_j,\tr\left(K\right)$ are the same as in Theorem \ref{thm:thm1}.
\end{lem}

\begin{proof}
Observe that $\exp\left(-Y^{T}SY\right)$ is an analytic function in the variables $s_{m,n}$ for $1\le m\le n\le N$ separately, then,
$$
\exp\left(-Y^{T}SY\right)=\sum_{K}\left(\prod_{1\le m\le n\le N}\frac{\left(s_{m,n}\right)^{k_{m,n}}}{k_{m,n}!}a_{k_{m,n}}\right),
$$
then for $1\le m=n\le N$,
$$
a_{k_{m,m}}=\frac{\partial^{k_{m,m}}\exp\left(-Y^{T}SY\right)}{\partial s_{m,m}^{k_{m,m}}}\bigg|_{S=0}=\left(-1\right)^{k_{m,m}}y_{m}^{2k_{m,m}},
$$
and for $1\le m<n\le N$,
$$
a_{k_{m,n}}=\frac{\partial^{k_{m,n}}\exp\left(-Y^{T}SY\right)}{\partial s_{m,n}^{k_{m,n}}}\bigg|_{S=0}=\left(-2\right)^{k_{m,n}}y_{m}^{k_{m,n}}y_{n}^{k_{m,n}}.
$$
We apply Theorem \ref{thmnDTaylor} and show that the error term $\to  0$ as $m \to \infty$
and conclude that
$$
\exp\left(-Y^{T}SY\right)=\sum_{K}\left(\prod_{1\le m\le n\le N}\frac{\left(-2s_{m,n}\right)^{k_{m,n}}}{k_{m,n}!}\right)2^{-\tr\left(K\right)}y_{1}^{k_{1}}y_{2}^{k_{2}}\dotsm y_{n}^{k_{n}}.
$$
\end{proof}

\begin{lem}\label{lem:2}
Let $H=\left(h_{j,k}\right)_{j,k=1}^{N}$ be an $N\times N$ complex matrix and $Z$ an $N\times1$ complex matrix $Z=\left(z_1,\dotsc,z_N\right)^T$ and $z_j=x_j+iy_j$, $x_j,y_j\in\mbb{R}$ for $j=1,\dotsc,N$, then,
\begin{equation}
\exp\left(-Z^{*}HZ\right)=\sum_{K}\prod_{j=1}^{N}\left(\overline{z_{j}}\right)^{r_{j}}z_{j}^{c_{j}}
\left(\prod_{1\le m,n\le N}\frac{\left(-h_{m,n}\right)^{k_{m,n}}}{k_{m,n}!}\right),
\label{eq:2.2}
\end{equation}
where $K,k_{m,n},r_{m},c_{n}$ are the same as in Theorem \ref{thm:thm2}.
\end{lem}

\begin{proof}
Observe that $\exp\left(-Z^{*}HZ\right)$ is analytic in the variables $h_{j,k}$, $j,k=1,\dotsc,N$ separately, then
$$
\exp\left(-Z^{*}HZ\right)=\sum_{K}\prod_{1\le m,n\le N}\frac{\left(h_{m,n}\right)^{k_{m,n}}}{k_{m,n}!}a_{k_{m,n}}
$$
and
$$
a_{k_{m,n}}=\frac{\partial^{k_{m,n}}\exp\left(-Z^{*}HZ\right)}{\partial h_{m,n}^{k_{m,n}}}\bigg|_{H=0}=\left(-\overline{z_{m}}z_{n}\right)^{k_{m,n}}
$$
for $k_{m,n}=0,1,\dotsc$ and $m,n=1,\dotsc,N$, then,
\begin{align*}
\exp\left(-Z^{*}HZ\right) &=\sum_{K}\prod_{1\le m,n\le N}\frac{\left(-h_{m,n}\right)^{k_{m,n}}\left(\overline{z_{m}}\right)^{k_{m,n}}\left(z_{n}\right)^{k_{m,n}}}{k_{m,n}!} \\
 &=\sum_{K}\prod_{j=1}^{N}\left(\overline{z_{j}}\right)^{r_{j}}z_{j}^{c_{j}}\left(\prod_{1\le m,n\le N}\frac{\left(-h_{m,n}\right)^{k_{m,n}}}{k_{m,n}!}\right).
\end{align*}
This completes the proof.
\end{proof}

\begin{rmk}
We have noticed that identities \ref{eq:2.1} and \ref{eq:2.2} can be proved formally without using Theorem \ref{thmnDTaylor}. We observe that for the former we have, 
\[
\begin{aligned} & \exp\left(-Y^{T}SY\right)=\exp\left(-\sum_{i,j=1}^{N}s_{i,j}y_{i}y_{j}\right)=\exp\left(-\sum_{i=1}^{N}s_{i,i}y_{i}^{2}-2\sum_{1\le i<j\le N}^{N}s_{i,j}y_{i}y_{j}\right)\\
& =\left(\prod_{i=1}^{N}\sum_{k_{i,i}=0}^{\infty}\frac{\left(-s_{i,i}y_{i}^{2}\right)^{k_{i,i}}}{\left(k_{i,i}\right)!}\right)\cdot\left(\prod_{1\le m<n\le N}\sum_{k_{m,n}=0}^{\infty}\frac{\left(-2s_{m,n}y_{m}y_{n}\right)^{k_{m,n}}}{\left(k_{m,n}\right)!}\right)\\
& =\sum_{k_{1,1},k_{2,2},\dots,k_{N,N}=0}^{\infty}\frac{\left(-s_{1,1}y_{1}^{2}\right)^{k_{1,1}}}{\left(k_{1,1}\right)!}\frac{\left(-s_{2,2}y_{2}^{2}\right)^{k_{2,2}}}{\left(k_{2,2}\right)!}\dots\frac{\left(-s_{N,N}y_{i}^{2}\right)^{k_{N,N}}}{\left(k_{N,N}\right)!}\\
& \times\sum_{k_{1,2},k_{1,3},k_{1,N},k_{2,3},\dots,k_{2,N},\dots,k_{N-1,N}=0}^{\infty}\frac{\left(-2s_{1,2}y_{1}y_{2}\right)^{k_{1,2}}}{\left(k_{1,2}\right)!}\dots\frac{\left(-2s_{1,N}y_{1}y_{N}\right)^{k_{1,N}}}{\left(k_{1,N}\right)!}\\
& \times\frac{\left(-2s_{2,3}y_{2}y_{3}\right)^{k_{2,3}}}{\left(k_{2,3}\right)!}\dots\frac{\left(-2s_{2,N}y_{2}y_{N}\right)^{k_{2,N}}}{\left(k_{2,N}\right)!}\dots\frac{\left(-2s_{N-1,N}y_{N-1}y_{N}\right)^{k_{N-1,N}}}{\left(k_{N-1,N}\right)!}\\
& =\sum_{K}\left(\prod_{i=1}^{N}\left(\frac{y_{i}}{2}\right)^{k_{i,i}}\right)\prod_{1\le m\le n\le N}\frac{\left(-2s_{m,n}y_{m}y_{n}\right)^{k_{m,n}}}{\left(k_{m,n}\right)!}=\sum_{K}2^{-\mathrm{tr}(K)}\prod_{1\le m\le n\le N}\frac{\left(-2s_{m,n}\right)^{k_{m,n}}}{\left(k_{m,n}\right)!}\\
& \times y_{1}^{2k_{1,1}+k_{1,2}+\dots+k_{1,N}}y_{2}^{2k_{2,2}+k_{2,3}+\dots k_{2,N}}\dots y_{N-1}^{2k_{N-1,N-1}+k_{N-1,N}}y_{N}^{2k_{N,N}}\\
& =\sum_{K}\left(\prod_{1\le m\le n\le N}\frac{\left(-2s_{m,n}\right)^{k_{m,n}}}{k_{m,n}!}\right)2^{-\mbox{tr}\left(K\right)}y_{1}^{k_{1}}y_{2}^{k_{2}}\cdots y_{N}^{k_{N}},
\end{aligned}
\]
whereas for the latter we have,
\[
\begin{aligned} & \exp\left(-Z^{*}HZ\right)=\exp\left(-\sum_{m,n=1}^{N}h_{m,n}\overline{z_{m}}z_{n}\right)=\prod_{1\le m,n\le N}\exp\left(-h_{m,n}\overline{z_{m}}z_{n}\right)\\
& =\prod_{1\le m,n\le N}\sum_{k_{m,n}=0}^{\infty}\frac{\left(-h_{m,n}\overline{z_{m}}z_{n}\right)^{k_{m,n}}}{\left(k_{m,n}\right)!}=\sum_{K}\prod_{1\le m,n\le N}\frac{\left(-h_{m,n}\right)^{k_{m,n}}}{\left(k_{m,n}\right)!}\\
& \times\left(\overline{z_{1}}\right)^{k_{1,1}+k_{1,2}+\dots+k_{1,N}}\dots\left(\overline{z_{N}}\right)^{k_{N,1}+k_{N,2}+\dots+k_{N,N}}z_{1}^{k_{1,1}+k_{2,1}+\dots+k_{N,1}}\dots\\
& \times z_{N}^{k_{1,N}+k_{2,N}+\dots+k_{N,N}}=\sum_{K}\left(\prod_{1\le m,n\le N}\frac{\left(-h_{m,n}\right)^{k_{m,n}}}{\left(k_{m,n}\right)!}\right)\left(\overline{z_{1}}\right)^{r_{1}}\dots\left(\overline{z_{N}}\right)^{r_{N}}z_{1}^{c_{1}}\dots z_{N}^{c_{N}}\\
& =\sum_{K}\left(\prod_{j=1}^{N}\left(\overline{z_{j}}\right)^{r_{j}}z_{j}^{c_{j}}\right)\left(\prod_{1\le m,n\le N}\frac{\left(-h_{m,n}\right)^{k_{m,n}}}{k_{m,n}!}\right).
\end{aligned}
\]

\end{rmk}

\begin{lem}\label{lem:4}	
For all $m,n\in\mbb{Z}^+$ and $z\in\mbb{C}$ we have
\begin{equation}
\left|H_{m,n}\left(\overline{z},z\right)\right|\le e^{|z|^{2}}\sqrt{m!\cdot n!}.
\label{eq:a-1}
\end{equation}
\end{lem}

\begin{proof}
From Equation \eqref{eq:1.9} we get
\begin{gather*}
\pi e^{-|z|^{2}} H_{m,n}\left(\overline{z},z\right)|
\le\int_{\mbb{R}^{2}}|w|^{m}\cdot|w|^{n}\exp\left\{-\left(r^{2}+s^{2}\right)\right\} drds \\
\begin{aligned}
&\le\sqrt{\int_{\mbb{R}^{2}}|w|^{2m}\exp\left\{-\left(r^{2}+s^{2}\right)\right\}drds}\sqrt{\int_{\mbb{R}^{2}}|w|^{2n}\exp\left\{-\left(r^{2}+s^{2}\right)\right\} drds} \\
&=\sqrt{\int_{\mbb{R}^{2}}(r^{2}+s^{2})^{m}\exp\left\{ -\left(r^{2}+s^{2}\right)\right\} drds}\sqrt{\int_{\mbb{R}^{2}}(r^{2}+s^{2})^{n}\exp\left\{ -\left(r^{2}+s^{2}\right)\right\} drds} \\
&=\pi\sqrt{m!\cdot n!},
\end{aligned}
\end{gather*}
which gives \eqref{eq:a-1}.
\end{proof}

We now present our proof of Theorem \ref{thm:thm1}.

\begin{proof}[Proof of Theorem \ref{thm:thm1}]
First we observe that
$$
\|S\|^{2}=\sum_{m,n=1}^{N}\left|s_{m,n}\right|^{2}=\tr\left(SS^T\right)=\sum_{j=1}^N\lambda_j^2,
$$
where $\lambda_{j}$, $j=1,\dotsc,N$ are the eigenvalues of $S$, then the matrix $I_{N}+S$ is positive definite and thus $\left(I_N+S\right)^{-1}$ exists and it is positive definite. It is clear that,
\begin{gather*}
\det\left(I_N+S\right)^{-\frac{1}{2}}\exp\left(X^{T}S\left(I_N+S\right)^{-1}X\right) \\
=\det\left(I_N+S\right)^{-\frac{1}{2}}\exp\left(-X^{T}\left(I_N+S\right)^{-1}X+X^TX\right),
\end{gather*}
then \eqref{eq:1.3} is equivalent to
\begin{gather*}
\det\left(I_N+S\right)^{-\frac{1}{2}}\exp\left(-X^{T}\left(I_{N}+S\right)^{-1}X\right) \\
=\sum_{K}\left(\prod_{1\le m\le n\le N}\frac{\left(s_{m,n}\right)^{k_{m,n}}}{2^{k_{m,n}}k_{m,n}!}\right)
 2^{-\tr\left(K\right)}\psi_{k_{1}}\left(x_{1}\right)\dotsb\psi_{k_{N}}\left(x_{N}\right),
\end{gather*}
where $\psi_n\left(x\right)=e^{-x^2}H_n\left(x\right)$. Applying the multivariate normal integral \cite{Bel}
\begin{align}
\int_{\mbb{R}^{N}}\exp\left(-X^{T}AX+2iB^{T}X\right)\prod_{j=1}^{N}dx_{j}
&=\sqrt{\frac{\pi^{N}}{\det A}}e^{-B^{T}A^{-1}B},
\label{eq:2.8}
\end{align}
where $N\in\mbb{N}$, $A$ is an $N\times N$ real symmetric positive definite matrix and $B$, $X$ are $N\times1$ real matrices, then using Lemma \ref{lem:1} and \eqref{eq:1.8} we get
\begin{gather*}
\det\left(I_{N}+S\right)^{-\frac{1}{2}}\exp\left(-X^{T}\left(I_N+S\right)^{-1}X\right) \\
=\pi^{-\frac{N}{2}}\int_{\mbb{R}^N}\exp\left(-Y^{T}\left(I_{N}+S\right)Y+2iX^{T}Y\right)\prod_{j=1}^Ndy_j \\
=\pi^{-\frac{N}{2}}\int_{\mbb{R}^N}\exp\left(-Y^{T}Y+2iX^{T}Y-Y^{T}SY\right)\prod_{n=1}^{N}dy_{n}\\
=\pi^{-\frac{N}{2}}\sum_{K}\left(\prod_{1\le m\le n\le N}\frac{\left(-2s_{m,n}\right)^{k_{m,n}}}{k_{m,n}!}\right)2^{-\tr\left(K\right)} \\
\times\int_{\mbb{R}^N}\exp\left(-Y^{T}Y+2iX^{T}Y\right)\prod_{n=1}^{N}y_{n}^{k_{n}}dy_{n} \\
=\sum_{K}\left(\prod_{1\le m\le n\le N}\frac{\left(-2s_{m,n}\right)^{k_{m,n}}}{k_{m,n}!}\right)2^{-\tr\left(K\right)} \\
\times\frac{1}{\left(-2i\right)^{\sum\limits_{j=1}^{N}k_{j}}}\,\psi_{k_{1}}\left(x_1\right)\dotsm\psi_{k_n}\left(x_n\right),
\end{gather*}
where the exchange order of summation and integration is valid because of
$$
\int_{\mbb{R}^N}\exp\left(-Y^{T}Y-Y^{T}SY\right)\prod_{n=1}^{N}dy_{n}<\infty
$$
and an application of Fubini's theorem. Observe that
$$
\sum_{j=1}^{N}k_{j}=2\sum_{1\le m\le n\le N}k_{m,n},
$$
then
\begin{gather*}
\det\left(I_N+S\right)^{-\frac{1}{2}}\exp\left(-X^{T}\left(I_N+S\right)^{-1}X\right) \\
=\sum_{K}\left(\prod_{1\le m\le n\le N}\frac{\left(\frac{s_{m,n}}{2}\right)^{k_{m,n}}}
{k_{m,n}!}\right)2^{-\tr\left(K\right)}\psi_{k_{1}}\left(x_{1}\right)\dotsm\psi_{k_{n}}\left(x_{n}\right),
\end{gather*}
which proves Theorem \ref{thm:thm1}.
\end{proof}

\begin{proof}[Proof of Theorem \ref{thm:thm2}]
Since $H$ is an $N\times N$ Hermitian matrix, then $H$ can be factored into
$$
H=U\Lambda U^*
$$
where $U$ is an $N\times N$ unitary matrix while $\Lambda$ is an $N\times N$ real diagonal, say $\Lambda=\diag\left\{\lambda_1,\dotsc,\lambda_N\right\}$, thus,
$$
\|H\|^2=\tr\left(HH^*\right)=\sum_{j=1}^N\lambda_j^2<1.
$$
Hence $I_N+H$ is a positive definite Hermitian matrix and thus it is invertible. It is clear that \eqref{eq:1.5} is the same as
\begin{gather*}
\det\left(I_{N}+H\right)^{-1}\exp\left(Z^{*}\left(I_{N}+H\right)^{-1}Z\right) \\
=\sum_{K}\prod_{1\le m,n\le N}\frac{\left(h_{m,n}\right)^{k_{m,n}}}{k_{m,n}!}\psi_{r_1,c_1}
\left(\overline{w_1},w_1\right)\dotsm\psi_{r_N,c_N}\left(\overline{w_N},w_N\right),
\end{gather*}
where $\psi_{\alpha,\beta}\left(z_1,z_2\right)=e^{-z_1z_2}H_{\alpha,\beta}\left(z_1,z_2\right)$. Set $\Gamma=\left(I_N+H\right)^{-1}$ to be an $N\times N$ positive definite Hermitian matrix, $\mu=\left(0,\dotsc,0\right)^{T}$ an $N\times1$ complex matrix and $C=\left(0\right)_{j,k=1}^{N}$ in the character function complex normal integral to get \cite{Cnd}
\begin{equation}
\begin{gathered}
\int_{\mbb{R}^{2N}}\exp\left(-Z^{*}\left(I_{N}+H\right)Z+2i\re\left(W^{*}Z\right)\right)\prod_{j=1}^{N}dx_{j}dy_{j} \\
=\frac{\pi^{N}\exp\left(-W^{*}\left(I_{N}+H\right)^{-1}W\right)}{\det\left(I_{N}+H\right)}
\end{gathered}
\label{eq:2.9}
\end{equation}
where $W=\left(w_1,\dotsc,w_N\right)^T$ is an $N\times1$ complex matrix. From Lemma \ref{lem:2} to get
\begin{gather*}
\frac{\exp\left(-W^{*}\left(I_{N}+H\right)^{-1}W\right)}{\det\left(I_{N}+H\right)} \\
\begin{aligned}
&=\pi^{-N}\int_{\mbb{R}^{2N}}\exp\left(-Z^{*}Z+2i\re\left(W^{*}Z\right)-Z^{*}HZ\right)\prod_{j=1}^{N}dx_{j}dy_{j} \\
&=\pi^{-N}\sum_{K}\left(\prod_{1\le m,n\le N}\frac{\left(-h_{m,n}\right)^{k_{m,n}}}{k_{m,n}!}\right) \\
&\quad\times\int_{\mbb{R}^{2N}}\exp\left(-Z^{*}Z+2i\re\left(W^{*}Z\right)\right)\prod_{j=1}^N\left(\overline{z_{j}}\right)^{r_j}z_j^{c_j}dx_jdy_j. \\
&=\sum_{K}\prod_{1\le m,n\le N}\frac{\left(h_{m,n}\right)^{k_{m,n}}}{k_{m,n}!}\psi_{r_1,c_1}\left(\overline{w_1},w_1\right)\dotsm\psi_{r_N,c_N}\left(\overline{w_N},w_N\right),
\end{aligned}
\end{gather*}
and the exchange order of summation and integration can be verified using Fubini's theorem.
\end{proof}

\begin{proof}[Proof of Theorem \ref{thm2b}]
By Lemma \ref{lem:4} we have
$$
\left|H_{r_1,c_1}\left(\overline{w_1},w_1\right)\cdots H_{r_N,c_N}\left(\overline{w_N},w_N\right)\right|
\le\exp\left(\sum_{i=1}^N|w_i|^2\right)\sqrt{\prod_{i=1}^Nr_i!\cdot\prod_{i=1}^{N}c_i!}
$$
and
\begin{gather*}
\left|\sum_{K}\prod_{1\le j,\ell\le N}\frac{\left(h_{j,\ell}\right)^{k_{j,\ell}}}{k_{j,\ell}!}H_{r_{1},c_{1}}
\left(\overline{w_{1}},w_{1}\right)\dotsm H_{r_{N},c_{N}}\left(\overline{w_{N}},w_{N}\right)\right| \\
\begin{aligned}
& \le\exp\left(\sum_{i=1}^{N}|w_{i}|^{2}\right)\sum_{K}\left\{\prod_{1\le j,\ell\le N}\frac{\left(\left|h_{j,\ell}\right|\right)^{k_{j,\ell}}
\prod\limits_{i=1}^{N}r_{i}!}{k_{j,\ell}!}\right\}^{\frac{1}{2}}
\left\{\prod_{1\le j,\ell\le N}\frac{\left(\left|h_{j,\ell}\right|\right)^{k_{j,\ell}}\prod\limits_{i=1}^{N}c_{i}!}{k_{j,\ell}!}\right\}^{\frac{1}{2}} \\
& \le\exp\left(\sum_{i=1}^{N}|w_{i}|^{2}\right)\left\{\sum_{K}\prod_{1\le j,\ell\le N}
\frac{\|H\|_{\infty}^{k_{j,\ell}}\prod\limits_{i=1}^{N}r_{i}!}{k_{j,\ell}!}\right\}
\left\{\sum_{K}\prod_{1\le j,\ell\le N}\frac{\|H\|_{\infty}^{k_{j,\ell}}\prod\limits_{i=1}^{N}c_{i}!}{k_{j,\ell}!}\right\}
\end{aligned}
\end{gather*}
by applying the Cauchy--Schwarz inequality. We observe that
\begin{gather*}
\sum_{K}\prod_{1\le j,\ell\le N}\frac{\|H\|_{\infty}^{k_{j,\ell}}\prod\limits_{i=1}^{N}r_i!}{k_{j,\ell}!}
=\sum_{K}\prod_{i=1}^{N}\binom{r_i}{k_{i,1},\dotsc,k_{i,N}}\,\|H\|_{\infty}^{\sum\limits_{\ell=1}^{N}k_{i,\ell}} \\
=\prod_{i=1}^{N}\sum_{r_i\ge0}\left(N\|H\|_{\infty}\right)^{r_i}=\left(1-N\|H\|_{\infty}\right)^{-N}
\end{gather*}
and
\begin{gather*}
\sum_{K}\prod_{1\le j,\ell\le N}\frac{\|H\|_{\infty}^{k_{j,\ell}}\prod\limits_{i=1}^{N}c_{i}!}{k_{j,\ell}!}
=\sum_{K}\prod_{i=1}^{N}\binom{c_{i}}{k_{1,i},\dotsc,k_{N,i}}\,\|H\| _{\infty}^{\sum\limits_{\ell=1}^{N}k_{\ell,i}} \\
=\prod_{i=1}^{N}\sum_{c_i\ge0}\left(N\|H\|_{\infty}\right)^{c_i}=\left(1-N\|H\|_{\infty}\right)^{-N},
\end{gather*}
then,
\begin{gather*}
\left|\sum_{K}\prod_{1\le j,\ell\le N}\frac{\left(h_{j,\ell}\right)^{k_{j,\ell}}}{k_{j,\ell}!}\,
H_{r_{1},c_{1}}\left(\overline{w_1},w_1\right)\dotsm H_{r_N,c_N}\left(\overline{w_N},w_N\right)\right| \\
\le\exp\left(\sum_{i=1}^N\left|w_i\right|^2\right)\left(1-N\|H\|_{\infty}\right)^{-2N}.
\end{gather*}
Hence the series on the right-hand side in \eqref{eq:a-2} converges uniformly and absolutely for $W$ in any compact subset of $\mbb{C}^N$ and $H$ in any compact subset of $B$.

Clearly, $\det\left(I_{N}+H\right)$ is a polynomial in variables $h_{j,\ell}$. For $H\in\mbb{C}^{N}$ we have \cite{Horn:Johnson}
$$
\left\|H^*\right\|_2=\|H\|_2
\le\sqrt{N}\,\|H\|_{\infty}<\frac1{\sqrt{N}}\le1.
$$
Then
$$
H\left(I_N+H\right)^{-1}=\sum_{m=1}^{\infty}(-1)^{m-1}H^{m}
$$
converges in norm $\|\cdot\|_2$, and
$$
\left\|H\left(I_{N}+H\right)^{-1}\right\|_2
\le\frac{\sqrt{N}\,\|H\|_{\infty}}{1-\sqrt{N}\,\|H\|_{\infty}},\
\left|W^*H\left(I_N+H\right)^{-1}W\right|\le\frac{\sqrt{N}\,\|H\|_{\infty}}{1-\sqrt{N}\,\|H\|_{\infty}}\,\|W\|^2.
$$
Consequently, $\exp\left(W^{*}H\left(I_{N}+H\right)^{-1}W\right)$ also converges absolutely and uniformly for $W$ in any compact subset of $\mbb{C}^{N}$ and $H$ in any compact subset of $B$. Let
\begin{gather*}
F(H,W)=\exp\left(W^*H\left(I_N+H\right)^{-1}W\right)-\det\left(I_N+H\right) \\
\times\sum_{K}\prod_{1\le j,\ell\le N}\frac{\left(h_{j,\ell}\right)^{k_{j,\ell}}}{k_{j,\ell}!}
H_{r_1,c_1}\left(\overline{w_{1}},w_{1}\right)\dotsm H_{r_N,c_N}\left(\overline{w_N},w_N\right).
\end{gather*}
Then for any fixed $W\in\mbb{C}^N$, $F(H,W)$ is analytic in variables $h_{j,k}$ in $B$, and $F\left(H_0,W\right)=0$ for any Hermitian matrix $H_0\in B$ by Theorem \ref{thm:thm2}.

Let us introduce a new coordinate system $u_{j,j}$, $u_{\ell,k}$, $v_{\ell,k}$, $1\le j,k,\ell\le N$, $\ell<k$ such that
$$
h_{j,j}=u_{j,j},\ h_{\ell,k}=u_{\ell,k}+iv_{\ell,k},\ \ell<k,\quad h_{k,\ell}=u_{\ell,k}-iv_{\ell,k},\ \ell>k.
$$
Since this is an invertible linear transformation, any function analytic in $h_{j,k}$, $1\le j,k\le N$ is also analytic in $u_{j,j}$, $u_{\ell,k}$, $v_{\ell,k}$, $1\le j,k,\ell\le N$, $\ell<k$ and vice versa. Furthermore, for any Hermitian matrix $H_{0}\in B$ and $\delta_{j,k}>0$, $1\le j,k\le N$ such that \eqref{eq:a-3} and \eqref{eq:a-4} and $D(H_{0},\delta)\subset B$ are satisfied, then $F(H,W)$ can be expanded into a convergent power series in variables $u_{j,j}$, $u_{\ell,k}$, $v_{\ell,k}$, $1\le j,k,\ell\le N$, $\ell<k$ at $u_{j,j}^{(0)}$, $u_{\ell,k}^{(0)}$, $v_{\ell,k}^{(0)}$, $1\le j,k,\ell\le N$, $\ell<k$ on $D\left(H_0,\delta\right)$. Clearly, $D\left(H_0,\delta\right)$ contains the following set $S$:
$$
-\delta_{j,j}<u_{j,j}-u_{j,j}^{(0)}<\delta_{j,j},\ -\delta_{\ell,k}<u_{\ell,k}-u_{\ell,k}^{(0)}<\delta_{\ell,k},\ -\delta_{\ell,k}<v_{\ell,k}-v_{\ell,k}<\delta_{\ell,k},
$$
where $1\le j,k,\ell\le N$, $\ell<k$, and $H$ is Hermitian on $S$. From Theorem \ref{thm:thm2} we know that $F(H,W)=0$ on $S$. Hence, all the coefficients in this power series expansion of $F(H,W)$ in variables
$u_{j,j}$, $u_{\ell,k}$, $v_{\ell,k}$, $1\le j,k,\ell\le N$, $\ell<k$ at $u_{j,j}^{(0)}$, $u_{\ell,k}^{(0)}$, $v_{\ell,k}^{(0)}$, $1\le j,k,\ell\le N$, $\ell<k$ must be zeros. Thus, $F(H,W)=0$ holds on $D\left(H_0,\delta\right)$, which is the same as \eqref{eq:a-2} in $D\left(H_0,\delta\right)$.	
\end{proof}

We now come to the proof of Corollary \ref{cor:cor1}.

\begin{proof}[Proof of Corollary \ref{cor:cor1}]
Let $W=\left(\rho_1e^{i\theta_1},\dotsc,\rho_Ne^{i\theta_N}\right)^T$ such that $\rho_j>0$, $\theta_j\in\mbb{R}$ for $j=1,\dotsc,N$ in \eqref{eq:1.5} to get
\begin{gather*}
\det\left(I_N+H\right)^{-1}\exp\left(W^{*}H\left(I_N+H\right)^{-1}W\right) \\
\begin{aligned}
&=\sum_{K}\prod_{1\le m,n\le N}\frac{\left(h_{m,n}\right)^{k_{m,n}}}{k_{m,n}!}\,
H_{r_1,c_1}\left(\overline{w_1},w_1\right)\dotsm H_{r_N,c_N}\left(\overline{w_N},w_N\right) \\
&=\sum_{K}\prod_{1\le m,n\le N}\frac{\left(h_{m,n}\right)^{k_{m,n}}}{k_{m,n}!}\,H_{r_1,c_1}
\left(\rho_1e^{i\theta_1},\rho_1e^{-i\theta_1}\right)\dotsm H_{r_N,c_N}\left(\rho_Ne^{i\theta_N},\rho_N e^{-i\theta_N}\right) \\
&=\sum_{K}\prod_{m=1}^N\prod_{n=1}^N\left(-h_{m,n}\right)^{k_{m,n}}\binom{c_m}{k_{1,m},\dotsc,k_{N,m}}
\,L_{c_m}^{\left(r_m-c_m\right)}\left(\rho_m^2\right)\left(\rho_{m}e^{i\theta_m}\right)^{r_m-c_m}.
\end{aligned}
\end{gather*}
This establishes \ref{eq:1.7}.

To prove \ref{eq:1.7a}, we let
 \[
 H=\left(\begin{array}{cc}
 0 & a\\
 b & 0
 \end{array}\right),\ W=\left(\begin{array}{c}
 x\\
 y
 \end{array}\right),\quad |a|,|b|<\frac{1}{4},\ a,b,x,y\in\mathbb{R},
 \]
 in \eqref{eq:1.7}, then
 \[
 \begin{aligned} & \det\left(I_{N}+H\right)^{-1}\exp\left(W^{*}H\left(I_{N}+H\right)^{-1}W\right)=\frac{\exp\left(-\frac{abx^{2}-xy(a+b)+aby^{2}}{1-ab}\right)}{1-ab}\\
 & =\sum_{k_{1,2},k_{2,1}=0}^{\infty}(-a)^{k_{1,2}}x^{k_{2,1}-k_{1,2}}L_{k_{1,2}}^{\left(k_{2,1}-k_{1,2}\right)}\left(x^{2}\right)(-b)^{k_{2,1}}y^{k_{1,2}-k_{2,1}}L_{k_{2,1}}^{\left(k_{1,2}-k_{2,1}\right)}\left(y^{2}\right)\\
 & =\sum_{j=0}^{\infty}\left(ab\right)^{j}L_{j}^{\left(0\right)}\left(x^{2}\right)L_{j}^{\left(0\right)}\left(y^{2}\right)+\sum_{0\le j<k<\infty}(-a)^{j}x^{k-j}L_{j}^{\left(k-j\right)}\left(x^{2}\right)(-b)^{k}y^{j-k}L_{k}^{\left(j-k\right)}\left(y^{2}\right)\\
 & +\sum_{0\le k<j<\infty}(-a)^{j}x^{k-j}L_{j}^{\left(k-j\right)}\left(x^{2}\right)(-b)^{k}y^{j-k}L_{k}^{\left(j-k\right)}\left(y^{2}\right)\\
 & =\sum_{j=0}^{\infty}\left(ab\right)^{j}L_{j}^{\left(0\right)}\left(x^{2}\right)L_{j}^{\left(0\right)}\left(y^{2}\right)+\sum_{0\le j<k<\infty}\frac{j!a^{j}b^{k}}{k!}\left(xy\right)^{k-j}L_{j}^{\left(k-j\right)}\left(x^{2}\right)L_{j}^{\left(k-j\right)}\left(y^{2}\right)\\
 & +\sum_{0\le k<j<\infty}\frac{k!a^{j}b^{k}}{j!}(xy)^{j-k}L_{k}^{\left(j-k\right)}\left(x^{2}\right)L_{k}^{\left(j-k\right)}\left(y^{2}\right)\\
 & =\sum_{j=0}^{\infty}\left(ab\right)^{j}L_{j}^{\left(0\right)}\left(x^{2}\right)L_{j}^{\left(0\right)}\left(y^{2}\right)+\sum_{0\le j<k<\infty}\frac{a^{j}b^{k}}{j!k!}\left(xy\right)^{k+j}C_{j}\left(k;x^{2}\right)C_{j}\left(k;y^{2}\right)\\
 & +\sum_{0\le k<j<\infty}\frac{a^{j}b^{k}}{j!k!}(xy)^{j+k}C_{k}\left(j;x^{2}\right)C_{k}\left(j;y^{2}\right),
 \end{aligned}
 \]
 where we have applied 
 \[
 L_{n}^{(x-n)}(a)=\frac{(-a)^{n}}{n!}C_{n}(x;a).
 \]
 Hence, for $x,y\neq0$, for $u,v$ sufficiently small, we let 
 \[
 a=\frac{u}{xy},\quad b=\frac{v}{xy}
 \]
 to get
 \[
 \begin{aligned} & \frac{x^{2}y^{2}}{x^{2}y^{2}-uv}\exp\left(-\frac{x^{2}y^{2}\left(x^{2}uv-x^{2}y^{2}(u+v)+y^{2}uv\right)}{x^{2}y^{2}-uv}\right)\\
 & =\sum_{j=0}^{\infty}\left(\frac{uv}{x^{2}y^{2}}\right)^{j}L_{j}^{\left(0\right)}\left(x^{2}\right)L_{j}^{\left(0\right)}\left(y^{2}\right)+\sum_{0\le j<k<\infty}\frac{u^{j}v^{k}}{j!k!}C_{j}\left(k;x^{2}\right)C_{j}\left(k;y^{2}\right)\\
 & +\sum_{0\le j<k<\infty}\frac{u^{k}v^{j}}{j!k!}C_{j}\left(k;x^{2}\right)C_{j}\left(k;y^{2}\right)=\sum_{j=0}^{\infty}\left(\frac{uv}{x^{2}y^{2}}\right)^{j}L_{j}^{\left(0\right)}\left(x^{2}\right)L_{j}^{\left(0\right)}\left(y^{2}\right)\\
 & +\sum_{0\le j<k<\infty}\frac{\left(u^{j}v^{k}+u^{k}v^{j}\right)}{j!k!}C_{j}\left(k;x^{2}\right)C_{j}\left(k;y^{2}\right).
 \end{aligned}
 \]
 By
 \[
 \begin{aligned}\sum_{j=0}^{\infty}\left(\frac{uv}{x^{2}y^{2}}\right)^{j}L_{j}^{\left(0\right)}\left(x^{2}\right)L_{j}^{\left(0\right)}\left(y^{2}\right) & =\frac{x^{2}y^{2}}{x^{2}y^{2}-uv}\exp\left(-\frac{uv(x^{2}+y^{2})}{x^{2}y^{2}-uv}\right){}_{0}F_{1}\left(-,1,\frac{uvx^{3}y^{3}}{\left(x^{2}y^{2}-uv\right)^{2}}\right)\\
 & =\frac{x^{2}y^{2}}{x^{2}y^{2}-uv}\exp\left(-\frac{uv(x^{2}+y^{2})}{x^{2}y^{2}-uv}\right)I_{0}\left(2\frac{\sqrt{uvx^{3}y^{3}}}{x^{2}y^{2}-uv}\right)
 \end{aligned}
 \]
 we have 
 \[
 \begin{aligned} & \frac{xy}{xy-uv}\exp\left(-\frac{xy\left(xuv-xy(u+v)+yuv\right)}{xy-uv}\right)-\frac{xy}{xy-uv}\exp\left(-\frac{uv(x^{2}+y^{2})}{xy-uv}\right)I_{0}\left(2\frac{\sqrt{uv(xy)^{3/2}}}{xy-uv}\right)\\
 & =\sum_{0\le j<k<\infty}\frac{\left(u^{j}v^{k}+u^{k}v^{j}\right)}{j!k!}C_{j}\left(k;x\right)C_{j}\left(k;y\right).
 \end{aligned}
 \]
\end{proof}

\section{Miscellaneous results}\label{sec3}

\begin{thm}\label{lem:lem3}
Let $w=r+is$ with $r,s\in\mbb{R}$ and $z_{1},z_{2}\in\mbb{C}$, then we have the moment integral representation
\begin{equation}
e^{-z_1z_2}H_{m,n}\left(z_1,z_2\right)
=\frac{1}{\pi i^{m+n}}\int_{\mbb{R}^2}\overline{w}^mw^n
\exp\left\{-w\overline{w}+iz_1w+iz_2\overline{w}\right\} drds.
\label{eq:2.3}
\end{equation}
In particular we have
\begin{equation}
\label{eq:2.4}
\begin{gathered}
e^{-r_{1}r_{2}e^{i\left(\theta_1+\theta_2\right)}}H_{m,n}\left(r_{1}e^{i\theta_1},r_2e^{i\theta_2}\right) \\
=\frac{i}{2^{m+n+1}\sqrt{\pi}}\int_{0}^{2\pi}H_{n+m+1}\left(\frac{r_{1}e^{i\left(\theta_{1}+\phi\right)}+r_{2}e^{i\left(\theta_{2}-\phi\right)}}{2}\right) \\
\times\exp\left\{ -\frac{\left(r_{1}e^{i\left(\theta_{1}+\phi\right)}+r_{2}e^{i\left(\theta_{2}-\phi\right)}\right)^{2}}{4}+i\left(n-m\right)\phi\right\} d\phi,
\end{gathered}
\end{equation}
\begin{equation}
\begin{gathered}
e^{-r^{2}}H_{m,n}\left(re^{i\theta},re^{-i\theta}\right)
=\frac{ie^{i\left(m-n\right)\theta}}{2^{m+n+1}\sqrt{\pi}} \\
\times\int_{0}^{2\pi}H_{n+m+1}\left(r\cos\phi\right)\exp\left(-r^{2}\cos^{2}\phi+i\left(n-m\right)\phi\right)d\phi,
\end{gathered}
\label{eq:2.5}
\end{equation}
\begin{equation}
H_{m,n}\left(w_{1},w_{2}\right)
=\begin{cases}
\left(-1\right)^{n}n!w_{1}^{m-n}L_{n}^{\left(m-n\right)}\left(w_{1}w_{2}\right),\quad & m\ge n\\
\left(-1\right)^{m}m!w_{2}^{n-m}L_{m}^{\left(n-m\right)}\left(w_{1}w_{2}\right),\quad & n\ge m
\end{cases}
\label{eq:4}
\end{equation}
and
\begin{equation}
\label{eq:2.7}
\begin{gathered}
\frac{i}{2\sqrt{\pi}}\int_{0}^{2\pi}H_{n+m+1}\left(r\cos\phi\right)\exp\left(-r^{2}\cos^{2}\phi+i\left(n-m\right)\phi\right)d\phi \\
=\left(-1\right)^{n}2^{m+n}n!r^{m-n}e^{-r^{2}}L_{n}^{\left(m-n\right)}\left(r^{2}\right).
\end{gathered}
\end{equation}
\end{thm}

\begin{proof}
Let
$$
a_{m,n}=\frac1{\pi}\int_{\mbb{R}^2}\overline{w}^{m}w^{n}\exp\left\{-w\overline{w}+iz_{1}w+iz_{2}\overline{w}\right\} drds
$$
then,
$$
\sum_{m,n=0}^{\infty}a_{m,n}\frac{u^{m}}{m!}\frac{v^{n}}{n!}
=\exp\left(-z_{1}z_{2}\right)\exp\left(iz_{1}u+iz_{2}v+uv\right).
$$
Comparing the above expression with the generating function \eqref{eqGFHmn} proves Equation \eqref{eq:2.3}. Let $w=\rho e^{i\phi}$, $z_{1}=r_{1}e^{i\theta_{1}}$ and $z_{2}=r_{2}e^{i\theta_{2}}$ in \eqref{eq:2.3}, then
\begin{gather*}
e^{-z_{1}z_{2}}H_{m,n}\left(z_{1},z_{2}\right) \\
\begin{aligned}
&=\frac{1}{\pi i^{m+n}}\int_{\mbb{R}^{2}}\overline{w}^{m}w^{n}\exp\left\{ -w\overline{w}+iz_{1}w+iz_{2}\overline{w}\right\} drds \\
&=\frac{1}{\pi i^{m+n}}\int_{0}^{2\pi}e^{i\left(n-m\right)\phi}\left\{ \int_{0}^{\infty}\rho^{m+n+1}\exp\left(-\rho^{2}+i\rho\left(e^{i\phi}z_{1}+e^{-i\phi}z_{2}\right)\right)d\rho\right\} d\phi \\
&=\frac{i}{2^{m+n+1}\sqrt{\pi}}\int_{0}^{2\pi}H_{n+m+1}\left(\frac{r_{1}e^{i\left(\theta_{1}+\phi\right)}+r_{2}e^{i\left(\theta_{2}-\phi\right)}}{2}\right) \\
&\quad\times\exp\left\{ -\frac{\left(r_{1}e^{i\left(\theta_{1}+\phi\right)}+r_{2}e^{i\left(\theta_{2}-\phi\right)}\right)^{2}}{4}+i\left(n-m\right)\phi\right\} d\phi,
\end{aligned}
\end{gather*}
where we used a variant of \eqref{eq:1.8} in the last step. This gives \eqref{eq:2.4}. Let $z_{1}=re^{i\theta}$, $z_{2}=re^{-i\theta}$ in \eqref{eq:2.4} to get,
\begin{eqnarray}
\begin{gathered}
e^{-r^{2}}H_{m,n}\left(re^{i\theta},re^{-i\theta}\right)\\
=\frac{i}{2^{m+n+1}\sqrt{\pi}}\int_{0}^{2\pi}H_{n+m+1}\left(r\cos\left(\theta+\phi\right)\right)\exp\left(-r^{2}\cos^{2}\left(\theta+\phi\right)+i\left(n-m\right)\phi\right)d\phi,
\end{gathered}
\notag
\end{eqnarray}
and we establish \eqref{eq:2.5}. The identification \eqref{eq:4} is known and follows from \eqref{eq:1.1} and the representation of a Laguerre polynomial as a confluent hypergeometric polynomial. Finally \eqref{eq:2.7} follows from \eqref{eq:2.5} and \eqref{eq:4}.
\end{proof}

The next result develops mixed relations involving $2$D Hermite polynomials and Hermite polynomials.

\begin{thm}\label{lem:8}
Let $w_{1},w_{2},z\in\mbb{C}$ with $z\ne0$, $\rho>0$ and $\theta\in\mbb{R}$, then we have
\begin{gather}
H_{n}\left(\frac{w_{1}+w_{2}}{2}\right)=z^{n}\sum_{j=0}^{n}\binom{n}{j}H_{j,n-j}\left(zw_{1},\frac{w_{2}}{z}\right)z^{-2j},
\label{eq:5} \\
H_{n}\left(\frac{\rho\left(z+z^{-1}\right)}{2}\right)=\frac{n!}{\left(-\rho z\right)^{n}}
\sum_{j=0}^{n}\frac{\left(-\rho^{2}z^{2}\right)^{j}}{j!}L_{n-j}^{\left(2j-n\right)}\left(\rho^{2}\right),
\label{eq:6} \\
H_{n}\left(\rho\cos\theta\right)=\frac{n!}{\left(-\rho e^{i\theta}\right)^{n}}\sum_{j=0}^{n}
\frac{\left(-\rho^{2}e^{2i\theta}\right)^{j}}{j!}L_{n-j}^{\left(2j-n\right)}\left(\rho^{2}\right),
\label{eq:7} \\
\int_{0}^{2\pi}H_{n}\left(\rho\cos\theta\right)e^{-ik\theta}d\theta
=\begin{cases}
0\quad & 2\nshortmid\left(n+k\right)\\
\frac{2\pi n!\left(-1\right)^{\left(n-k\right)/2}\rho^{k}}{\left(\frac{n+k}{2}\right)!}
L_{\frac{n-k}{2}}^{\left(k\right)}\left(\rho^{2}\right)\  & 2\mid\left(n+k\right)
\end{cases}\label{eq:8}
\end{gather}
and
\begin{equation}
\int_{0}^{2\pi}\left(H_{n}\left(\rho\cos\theta\right)\right)^{2}\frac{d\theta}{2\pi}
=\frac{\left(n!\right)^2}{\rho^{2n}}\sum_{j=0}^{n}\frac{\rho^{2j}}{j!j!}\left(L_{n-j}^{\left(2j-n\right)}\left(\rho^2\right)\right)^2.
\label{eq:9}
\end{equation}
\end{thm}

\begin{proof}
Let $u=\alpha^2t$, $v=\beta^2t$, $z_1=2\beta w_1/\alpha$, $z_2=2\alpha w_2/\beta$ in \eqref{eqGFHmn} to obtain
\begin{gather*}
\sum_{m,\, n=0}^{\infty}\, H_{m,n}\left(\frac{2\beta w_1}
 {\alpha},\frac{2\alpha w_{2}}{\beta}\right)\frac{\alpha^{2m}}{m!}
 \;\frac{\beta^{2n}}{n!}t^{m+n} \\
\begin{aligned}
  &=\exp\left(-\left(\alpha\beta t\right)^{2}+2\alpha\beta t
  \left(w_{1}+w_{2}\right)\right)\\
 & =\sum_{n=0}^{\infty}H_{n}\left(w_{1}+w_{2}\right)
 \frac{\left(\alpha\beta t\right)^{n}}{n!}
\end{aligned}
\end{gather*}
and \eqref{eq:5} follows by equating like coefficients of powers of $t$. We use the parameter identification $z=\beta/\alpha$, $w_{1}= e^{i\theta}\rho\alpha/2\beta$, $w_2=e^{-i\theta}\rho\beta/2\alpha$ in \eqref{eq:5} and find that
\begin{gather*}
H_{n}\left(\frac{\rho\alpha}{2\beta}e^{i\theta}+\frac{\rho\beta}
{2\alpha}e^{-i\theta}\right)=\left(\frac{\beta}{\alpha}\right)^{n}
\sum_{j=0}^{n}\binom{n}{j}H_{j,n-j}\left(\rho e^{i\theta},\rho e^{-i
\theta}\right)\left(\frac{\alpha}{\beta}\right)^{2j} \\
\begin{aligned}
 & =\left(\frac{\beta}{\alpha}\right)^{n}\sum_{j=0}^{n}\binom{n}{j}
 \left(-1\right)^{n-j}\left(n-j\right)!\left(\rho e^{i\theta}\right)^{2j-n}
 L_{n-j}^{\left(2j-n\right)}\left(\rho^{2}\right)
 \left(\frac{\alpha}{\beta}\right)^{2j} \\
 & =\left(\frac{\beta}{\alpha}\right)^{n}\frac{n!}{\left(-\rho e^{i\theta}\right)^{n}}\sum_{j=0}^{n}\frac{\left(-\rho^{2}e^{2i\theta}\right)^{j}}{j!}L_{n-j}^{\left(2j-n\right)}\left(\rho^{2}\right)\left(\frac{\alpha}{\beta}\right)^{2j},
\end{aligned}
\end{gather*}
and \eqref{eq:6} follows. Equations \eqref{eq:7}--\eqref{eq:8} follow by taking $z=e^{i\theta}$ in \eqref{eq:6} and applying the Fourier orthogonality.
\end{proof}

We list more properties for the $2$D Hermite polynomials:

\begin{thm}\label{lem:7}
Let $z_{1},z_{2},w_{1},w_{2}\in\mbb{C}$ and $m,n$ are negative integers, then we have
\begin{align}
H_{m,n}\left(w_{1}-iw_{2},w_{1}+iw_{2}\right)
&=\frac{i^{m-n}}{2^{m+n}}\sum_{j,k=0}^{\min\left(m,n\right)}
\binom{m}{j}\binom{n}{k}\frac{H_{j+k}\left(w_{1}\right)
H_{m+n-j-k}\left(w_{2}\right)}{i^{j-k}},
\label{eq:1} \\
\frac{H_{m,n}\left(z_{1}+w_{1},z_{2}+w_{2}\right)}
{e^{w_{1}w_{2}+z_{1}w_{2}+z_{2}w_{1}}}
&=\sum_{j,k=0}^{\infty}\frac{\left(-w_{1}\right)^{j}
\left(-w_{2}\right)^{k}}{j!k!}H_{m+k,n+j}
\left(z_{1},z_{2}\right),
\label{eq:2} \\
H_{m,n}(0,0) &=\delta_{m,n}\left(-1\right)^{n}n!,
\label{eq:3}
\end{align}
\end{thm}

\begin{proof}
From
\begin{equation*}
\begin{split}
e^{-z_{1}z_{2}}H_{m,n}\left(z_{1},z_{2}\right)
&=\frac{1}{\pi i^{m+n}}\int_{\mbb{R}^{2}}\overline{w}^{m}w^{n} \\
&\quad\times\exp\left\{-w\overline{w}+iz_{1}w+iz_{2}\overline{w}\right\} drds \\
&=\frac{1}{\pi i^{m+n}}\sum_{j,k=0}^{\infty}\binom{m}{j}
 \binom{n}{k}\,i^{m-n-j+k}\int_{\mbb{R}^{2}}r^{j+k}s^{m+n-j-k} \\
 &\quad\times\exp\left\{-r^{2}-s^{2}+ir\left(z_{1}+z_{2}\right)+
 is\left(iz_{1}-iz_{2}\right)\right\} drds \\
 & =\frac{1}{\pi i^{m+n}}\sum_{j,k=0}^{\infty}\binom{m}{j}
 \binom{n}{k}\,i^{m-n-j+k}\int_{\mbb{R}}r^{j+k}e^{-r^{2}+ir
 \left(z_{1}+z_{2}\right)}dr \\
 &\quad\times\int_{\mbb{R}}s^{m+n-j-k}e^{-s^{2} +
 is\left(iz_{1}-iz_{2}\right)}dr \\
 & =\frac{i^{m-n}}{2^{m+n}}\sum_{j,k=0}^{\infty}\binom{m}{j}
 \binom{n}{k}\,i^{-j+k}H_{j+k}\left(\frac{z_{1}+z_{2}}{2}\right)
 H_{m+n-j-k}\left(\frac{z_{1}-z_{2}}{2}i\right).
\end{split}
\end{equation*}
or
$$
H_{m,n}\left(z_{1},z_{2}\right)=\frac{i^{m-n}}{2^{m+n}}\sum_{j,k=0}^{\infty}\binom{m}{j}\binom{n}{k}\,
i^{-j+k}H_{j+k}\left(\frac{z_{1}+z_{2}}{2}\right)H_{m+n-j-k}\left(\frac{z_{1}-z_{2}}{2}i\right).
$$
Let $z_{1}=w_{1}-iw_{2}$ and $z_{2}=w_{1}+iw_{2}$ we get \eqref{eq:1}.

From
\begin{gather*}
e^{-\left(z_{1}+w_{1}\right)\left(z_{2}+w_{2}\right)}H_{m,n}
 \left(z_{1}+w_{1},z_{2}+w_{2}\right)\\
\begin{aligned}
 & =\frac{1}{\pi i^{m+n}}\sum_{j,k=0}^{\infty}
 \frac{w_{1}^{j}w_{2}^{k}i^{j+k}}{j!k!}\int_{\mbb{R}^{2}}
 \overline{x}^{m+k}x^{n+j}\exp\left\{ -x\overline{x}+iz_{1}x
 +iz_{2}\overline{x}\right\} drds \\
 & =\sum_{j,k=0}^{\infty}\frac{w_{1}^{j}w_{2}^{k}
 \left(-1\right)^{j+k}}{j!k!}e^{-z_{1}z_{2}}H_{m+k,n+j}
 \left(z_{1},z_{2}\right),
\end{aligned}
\end{gather*}
that is \eqref{eq:2}. Formula \eqref{eq:3} follows from \eqref{eqGFHmn}.
\end{proof}

\section{$q$-analogues}\label{sec4}

We follow the notation for $q$-shifted factorials and $q$-series as in \cite{And:Ask:Roy}, \cite{Gas:Rah} and \cite{Ism}. The $2D$ $q$-Hermite polynomials are defined by \cite{Ism:Zha}
\begin{equation}
\label{eqHmnq}
\frac{H_{m,n}(z_1,z_2\mid q)}{(q;q)_m(q;q)_n}
=\sum_{k=0}^{m\wedge n}\frac{z_1^{m-k}z_2^{n-k}}{(q;q)_{m-k}(q;q)_{n-k}(q;q)_k}.
\end{equation}
In \cite{Ism:Zha} we also proved the generating function
\begin{equation}
\sum_{m,n=0}^\infty H_{m,n}\left(z_1,z_2\mid q\right)\frac{u^m\,v^n}{(q;q)_m(q;q)_n}
=\frac{(uv;q)_\infty}{\left(uz_1,vz_2;q\right)_\infty}.
\label{eqGFHq}
\end{equation}
We shall also use the Askey--Wilson integral \cite{Ism,Gas:Rah,And:Ask:Roy}
\begin{equation}
\label{eqAWI}
\int_0^\pi\frac{\left(e^{2i\theta},e^{-2i\theta};q\right)_\infty}{\prod\limits_{j=1}^4\left(t_je^{i\theta},t_je^{-i\theta};q\right)_\infty} d\theta
=\frac{2\pi\left(t_1t_2t_3t_4;q\right)_\infty}{(q;q)_\infty\,\prod\limits_{1\le j<k\le 4}\left(t_jt_k;q\right)_\infty},
\end{equation}
which holds for $\max\left\{\left|t_j\right|:1\le j\le 4\right\}<1$. The trigonometric moments of the $q$-Hermite weight function are \cite{Ism}
\begin{equation}
\label{eqqHmoments}
\int_0^{\pi}e^{2ij\theta}\left(e^{2i\theta},e^{-2i\theta};q\right)_{\infty}\frac{d\theta}{2\pi}
=\frac{(-1)^{j}}{(q;q)_{\infty}}\left(q^{\binom{j}{2}}+q^{\binom{-j}{2}}\right).
\end{equation}

\begin{thm}\label{thm:asc}
For $\left|rz_1\right|<1$, $\left|sz_2\right|<1$, we have the generating function
\begin{equation}
\begin{gathered}
\frac{(rs,rs ;q)_{\infty}}{\left(z_1z_2 rs;q\right)_{\infty}} \\
=\sum_{m_{1},m_{2}=0}^{\infty}\sum_{n_{1},n_{2}=0}^{\infty}
 \frac{H_{m_{1},n_{1}}\left(z_1,z_2\mid q\right)H_{m_{2},n_{2}}
 \left(z_1,z_2\mid q\right)}{\left(q;q\right)_{m_{1}}
 \left(q;q\right)_{m_{2}}\left(q;q\right)_{n_{1}}
 \left(q;q\right)_{n_{2}}}\,r^{m_1+m_2}s^{n_1+n_2} \\
 \times\frac{\left(q^{\binom{\left(m_{1}+n_{2}-n_{1}-m_2\right)/2}{2}}
 +q^{\binom{\left(n_{1}+m_{2}-m_{1}-n_{2}\right)/2}{2}}\right)}
 {\left(-1\right)^{\left(n_{1}+n_{2}-m_{1}-m_{2}\right)/2}},
\end{gathered}
\label{eq:asc}
\end{equation}
where the summation is over all the nonnegative integers such that $m_{1}+m_{2}-n_{1}-n_{2}$ is even.
\end{thm}

\begin{proof}
Multiply the generating functions \eqref{eqGFHq} with the $z$ variable being $z_1$, $z_2$, $z_1$, $z_2$ and set $u_1=re^{i\theta}$, $v_1=se^{-i\theta}$, $u_2=re^{-i\theta}$, $v_2=se^{i\theta}$. This gives
\begin{gather*}
\frac{(rs,rs;q)_\infty}
{\left(rz_1e^{i\theta},rz_1e^{-i\theta},sz_2e^{i\theta},sz_2e^{-i\theta};q\right)_\infty} \\
=\sum_{m_{1},m_{2},n_1,n_2=0}^{\infty}\frac{H_{m_{1},n_{1}}
\left(z_1,z_2\mid q\right)H_{m_{2},n_{2}}\left(z_1,z_2)\mid q\right)
 r^{m_1+m_2}s^{n_1+n_2}}{\left(q;q\right)_{m_1}
 \left(q;q\right)_{m_{2}}\left(q;q\right)_{n_{1}}\left(q;q\right)_{n_{2}}
 e^{i\theta\left(n_{1}+ m_2-m_{1}-n_{2}\right)}}.
\end{gather*}
Multiply the above generating function by $\left(e^{2i\theta},e^{-2i\theta};q\right)_\infty/(2\pi)$ and integrate over $\theta$ in $[-\pi,\pi]$ the apply the case $t_3=t_4=0$ of the Askey--Wilson integral \eqref{eqAWI} we find that
\begin{gather*}
2\,\frac{(rs,rs;q)_{\infty}}{\left(q;z_1z_2 rs;q\right)_{\infty}} \\
\begin{aligned}
 & =\sum_{m_{1},m_{2}=0}^{\infty}\sum_{n_{1},n_{2}=0}^{\infty}
 \frac{H_{m_{1},n_{1}}\left(z_1,z_2\mid q\right)H_{m_{2},n_{2}}
 \left(z_1,z_2\mid q\right)}{\left(q;q\right)_{m_{1}}
 \left(q;q\right)_{m_{2}}\left(q;q\right)_{n_{1}}
 \left(q;q\right)_{n_{2}}}\,r^{m_1+m_2}s^{n_1+n_2} \\
 &\quad\times \int_{-\pi}^{\pi}
 \frac{\left(e^{2i\theta},e^{-2i\theta};q\right)_{\infty}}
 {e^{i\theta\left(n_{1}+ m_2-m_{1}-n_{2}\right)}}\frac{d\theta}{2\pi},
\end{aligned}
\end{gather*}
which vanishes unless $m_2+n_1-m_1-n_2$ is even, in which case we get
\begin{gather*}
\frac{(rs,rs ;q)_{\infty}}{\left(q;z_1z_2rs;q\right)_{\infty}} \\
\begin{aligned}
&=\sum_{m_{1},m_{2}=0}^{\infty}\sum_{n_{1},n_{2}=0}^{\infty}
 \frac{H_{m_{1},n_{1}}\left(z_1,z_2\mid q\right)H_{m_{2},n_{2}}
 \left(z_1,z_2\mid q\right)}{\left(q;q\right)_{m_{1}}
 \left(q;q\right)_{m_{2}}\left(q;q\right)_{n_{1}}
 \left(q;q\right)_{n_{2}}}\,r^{m_1+m_2}s^{n_1+n_2} \\
&\quad\times\frac{\left(q^{\binom{\left(m_{1}+n_{2}-n_{1}-m_2\right)/2}{2}}
   +q^{\binom{(n_{1}+m_{2}-m_{1}-n_{2})/2}{2}}\right)}
   {\left(q;q\right)_{\infty}\left(-1\right)^{\left(n_{1}+n_{2}-m_{1}-m_{2}\right)/2}},
 \end{aligned}
\end{gather*}
where the summation is over all the nonnegative integers such that $m_{1}+m_{2}-n_{1}-n_{2}$ is even.
\end{proof}

\begin{thm}
Let $\left|x_{1}z_{1}\right|,\left|x_{1}z_{2}\right|,\left|x_{1}z_{3}\right|,\left|x_{1}z_{4}\right|<1$, then
\begin{gather}
\frac{\left(r_1s_1,r_1s_1,r_2s_2,r_2s_2,r_1s_1,r_2 s_2z_1z_2z_3z_4;q\right)_\infty}
{\left(r_1r_2z_1z_2,r_1r_2z_1z_3,r_1s_2z_1z_4,r_2s_1z_2z_3, s_1s_2z_2z_4,r_2s_2z_3z_4,q\right)_\infty} \\
\begin{aligned}
&=\sum_{m_j, n_j\ge 0, 1\le j\le 4}^\infty
    \frac{H_{m_{1},n_{1}}\left(z_1,z_2\right)\mid q)
    H_{m_2,n_2}\left(z_{1},z_2\mid q\right)}
    {\left(q;q\right)_{m_{1}}\left(q;q\right)_{n_1}\left(q;q\right)_{m_{2}}\left(q;q\right)_{n_2}}\,
    r_1^{m_1+m_2}r_2^{m_3+m_3}s_1^{n_1+n_2} s_2^{n_3+n_4} \\
&\quad\times\frac{H_{m_{3},n_{3}}\left(z_3,z_4\right)\mid q)
    H_{m_4,n_4}(z_{3},z_4|q)}
    {\left(q;q\right)_{m_{3}}\left(q;q\right)_{n_3}\left(q;q\right)_{m_{3}}\left(q;q\right)_{n_3}}
     (-1)^M\left[q^{\binom{M}{2}}+q^{\binom{-M}{2}}\right]
\end{aligned}
\label{eq:qks}
\end{gather}
where the summation is over all nonnegative integers $m_j,n_j$, $1\le j\le 4$ such that $m_1+n_2+m_3+n_4-n_1-m_2-n_3-m_4$ is even and $=2M$.
\end{thm}

\begin{proof}
Again we start with four cases of the generating function \eqref{eqGFHq} with parameters $\left(u_j,v_j\right)$, $1\le j\le 4$ and variables $z_1$, $z_2$, $z_1$, $z_2$, $z_3$, $z_4$, $z_3$, $z_4$, where $u_1=r_1e^{i\theta}$, $v_1=s_1e^{-i\theta}$, $u_2=r_1e^{-i\theta}$, $v_2=s_1e^{i\theta}$, $u_3=r_2e^{i\theta}$, $v_3=s_2e^{-i\theta}$, $u_4=r_2e^{-i\theta}$, $v_4=s_2 e^{i\theta}$. We multiply the four right-hand sides of \eqref{eqGFHq} by $\left(e^{2i\theta},e^{-2i\theta};q\right)_\infty/(2\pi)$ and integrate over $[-\pi,\pi]$. The use of the Askey--Wilson integral shows that the result is
\begin{eqnarray}
\frac{2\; (r_1s_1, r_1s_1, r_2 s_2, r_2s_2, r_1s_1, r_2 s_2z_1z_2z_3z_4;q)_\infty}{(q, r_1r_2z_1z_2,r_1r_2z_1z_3, r_1s_2z_1z_4, r_2s_1z_2z_3, s_1s_2z_2z_4, r_2s_2z_3z_4,q)_\infty}.
\notag
\end{eqnarray}
The rest of the proof is similar to the proof of Theorem \ref{thm:asc} and will be omitted.
\end{proof}




\section*{References}
\bibliographystyle{plainnat}
\bibliography{ismail-zhang-AAM}





\end{document}